\documentclass{lms}

\usepackage{amsmath}
\usepackage{amssymb}
\usepackage{amscd, cite} 
\usepackage{enumitem}

\numberwithin{equation}{section}
\newtheorem{theorem}{Theorem}[section]

\newtheorem{corollary}[theorem]{Corollary}

\newtheorem{proposition}[theorem]{Proposition}

\newnumbered{remark}[theorem]{Remark}
\newnumbered{example}[theorem]{Example}
\newnumbered{definition}[theorem]{Definition}
\newnumbered{hypothesis}[theorem]{Hypothesis}

\numberwithin{equation}{section}
\makeatletter
\let\c@theorem\c@equation
\makeatother
\renewcommand{\theequation}{\thetheorem}

\renewcommand{\leq}{\leqslant}
\renewcommand{\geq}{\geqslant}

\newcommand{\nin}{\notin}

\newcommand{\set}[1]{\{#1\}}
\newcommand{\setst}[2]{\{\,#1\mid #2 \}}
\newcommand{\gen}[1]{\langle #1 \rangle}

\newcommand{\tr}{\operatorname{tr}}
\renewcommand{\c}[2]{{ }^{#1}\!{#2}}
\newcommand{\bs}{\backslash}

\newcommand{\F}{\mathcal{F}}
\newcommand{\C}{\mathcal{C}}

\renewcommand{\phi}{\varphi}

\newcommand{\Hom}{\operatorname{Hom}}

\newcommand{\Aut}{\operatorname{Aut}}

\newcommand{\Syl}{\operatorname{Syl}}

\newcommand{\hyp}{\mathfrak{hyp}}
\newcommand{\foc}{\mathfrak{foc}}

\renewcommand{\mod}{\operatorname{mod}}

\title[Thompson-Lyons transfer lemma]{The Thompson-Lyons transfer lemma for fusion systems}
\author[J.~Lynd]{Justin Lynd}
\classno{20D20 (primary)}

\begin{document}
\maketitle
\begin{abstract}
A generalization of the Thompson transfer lemma and its various
extensions, most recently due to Lyons, is proven in the context of saturated
fusion systems. A strengthening of Alperin's fusion theorem is also given in
this setting, following Alperin's own ``up and down'' fusion.
\end{abstract}

The classical Thompson transfer lemma appeared as Lemma~5.38 in
\cite{Thompson1968}; for a $2$-perfect group $G$ with $S \in \Syl_2(G)$, it
says that if $T$ is a maximal subgroup of $S$ and $u$ is an involution in
$S-T$, then the element $u$ has a $G$-conjugate in $T$.  Thompson's lemma has
since been generalized in a number of ways.  Harada showed
\cite[Lemma~16]{Harada1968} that the same conclusion holds provided one takes
$u$ to be of least order in $S-T$. Unpublished notes of Goldschmidt
\cite{GoldschmidtTransfer} extended this to show that one may find an
$G$-conjugate of $u$ in $T$ which is extremal under the same conditions.  An
element $t \in S$ is said to be extremal in $S$ with respect to $G$ if $C_S(t)$
is a Sylow subgroup of $C_G(t)$. In other words, $\gen{t}$ is fully
$\F_S(G)$-centralized, where $\F_S(G)$ is the fusion system of $G$.  Later,
Thompson's result and its extensions were generalized to all primes via an
argument of Lyons \cite[Proposition~15.15]{GLS2}. We prove here a common
generalization of Lyons' extension and his similar transfer result
\cite[Chapter 2, Lemma 3.1]{GLS6} (which relaxes the requirement that $S/T$ be
cyclic) in the context of saturated fusion systems.

\begin{TL*}
Suppose that $\F$ is a saturated fusion system on the $p$-group $S$, and that
$T$ is a proper normal subgroup of $S$ with $S/T$ abelian. Let $u \in S-T$ and
let $\mathcal{I}$ be the set of fully $\F$-centralized $\F$-conjugates of $u$
in $S-T$. Assume 
\begin{enumerate}[label=\textup{(\arabic{*})}]
\item $u$ is of least order in $S-T$ and
\item the set of cosets $\mathcal{I}T = \set{vT \mid v \in \mathcal{I}}$ is
linearly independent in $\Omega_1(S/T)$.
\end{enumerate}
If $u \in \foc(\F)$, then $u$ has a fully $\F$-centralized $\F$-conjugate in
$T$.
\end{TL*}

Since $\hyp(\F) \leq \foc(\F)$, an immediate consequence of the Theorem is that
the conclusion holds in any fusion system with $\F = O^p(\F)$.  We refer to
\cite{AschbacherKessarOliver2011} and \cite{CravenTheory} for background on
fusion systems. Of course, both conditions in Theorem~TL are necessary. Take
$p=2$, $\F = \F_S(G)$, $S \in \Syl_2(G)$, and for (1): $G = SL_3(2)$, $T$ a
four group, and $u \in S-T$ of order $4$; while for (2): $G = SL_2(3)$, $T =
Z(S)$ and $u \in S-T$. 

In the classification of finite simple groups, direct applications of the
transfer map are primarily those arising from extensions of Thompson's lemma,
as well as Yoshida's Theorem on control of transfer.  The latter has an
analogue in fusion systems \cite{DiazGlesserParkStancu2011}. To the extent that
history is a guide, Theorem~TL should serve as a key ingredient in the analysis
of simple fusion systems at the prime $2$, especially in the component-type
portion of a program for the classification of simple $2$-fusion systems
outlined recently by M. Aschbacher; see
\cite[Sections~I.13-15]{AschbacherKessarOliver2011}. Indeed, we postpone
applications of Theorem~TL to the companion \cite{LyndCharacterization}, where
we classify certain $2$-fusion systems with an involution centralizer having a
component based on a dihedral $2$-group.

\section{Bisets and notation}
The proof we present is modeled on Lyons' argument mentioned above, and relies
on the transfer in saturated fusion systems.  Transfer in the fusion system
setting is defined by way of a characteristic biset associated to a saturated
fusion system, which is an $S$-$S$ biset $\Omega$ satisfying certain properties
first outlined by Linckelmann and Webb \cite{LinckelmannWebb}.  We motivate
these properties now while fixing notation; more details can be found in
\cite[Section~7.6]{CravenTheory}.

We compose maps left-to-right, sometimes writing applications of injective
homomorphisms in the exponent, as in $s^\phi$ when $s \in S$ and $\phi \in
\Hom_\F(\gen{s}, S)$. For groups $H$ and $K$, an $H$-$K$ biset is a set
$X = { }_HX_K$ together with an action of $H$ on the left and $K$ on the right,
such that $(hx)k = h(xk)$ for each $x \in X$, $h \in H$, and $k \in K$. An
$H$-$K$ biset $X$ may be regarded as a right $(H \times K)$-set via $x \cdot
(h,k) = h^{-1}xk$. A transitive biset is a biset with a single orbit under the
action of $H \times K$. 

Let $G$ be a finite group and $S$ a Sylow $p$-subgroup of $G$. Then $G
= { }_SG_S$ is an $S$-$S$ biset with left and right action given by
multiplication in $G$. The transitive subbisets of $G$ are the $S$-$S$ double
cosets of $G$. For $g \in G$, the stabilizer in $S \times S$ of $g$ is the
graph subgroup
\[
\Delta^{c_g}_{\c{g}{S} \cap S} = \setst{(s, s^g)}{s \in \c{g}{S} \cap S}
\]
of $S \times S$. Thus, the double coset 
\[
SgS \,\, \cong \,\, (S \times S)/\Delta^{c_g}_{\c{g}{S} \cap S}
\]
as a right $(S \times S)$-set. In general, for $P \leq S$ and an injective
group homomorphism $\phi\colon P \to S$, we write $\Delta_{P}^\phi$ for the
graph subgroup $\setst{(t, t^\phi)}{t \in P}$. 

Let $\psi\colon S \to A$ be a map to an abelian group $A$, and $T$ the kernel of
$\psi$. The transfer map $\tr_{S, \psi}^G\colon G \to A$ relative to $\psi$ is
the group homomorphism given by 
\[
\text{$u \tr_{S, \psi}^G = \prod_{h \in [G/S]} ([uh]^{-1}uh)\psi$ \quad\quad
for $u \in G$,} 
\] 
where $[G/S]$ is a set of representatives for the left cosets of $S$ in $G$ and
$[uh] \in [G/S]$ is the chosen representative for the coset $uhS$.  Restricting
$\tr_{S, \psi}^G$ to $S$ and decomposing this product by the left action of $S$
on the set of cosets $G/S$ gives rise to the Mackey decomposition (cf.
\cite[Lemma~15.13]{GLS2}) of the transfer map:
\[
\text{$u \tr_{S, \psi}^G = \prod_{g \in [S\bs G/S]} u \tr_{\c{g}{S} \cap S,\,
c_g\psi}^S$ \quad\quad for $u \in S$},  
\]
where $[S\bs G/S]$ is a set of representatives for the $S$-$S$ double cosets in
$G$ and $\tr_{\c{g}{S} \cap S, \,c_g\psi}^S$ is the transfer of the composite 
\[
\c{g}{S} \cap S \xrightarrow{c_g} S \cap S^g \xrightarrow{\psi|_{S \cap S^g}} (S
\cap S^g)T/T. 
\]
Thus, $\tr_{S,\, \psi}^G\!|_S$ is determined by the collection of morphisms
$\{c_g\colon \c{g}{S} \cap S \to S \cap S^g \mid g \in [S\bs G/S]\}$ in
$\F_S(G)$.

Let $P \leq S$.  For an $S$-$S$ biset $X$, denote by ${ }_SX_P$ the set $X$
considered as an $S$-$P$ biset upon restriction on the right to $P$. More
generally, for an injective group homomorphism $\phi\colon P \to S$, denote by
${ }_SX_\phi$ the $S$-$P$ biset with action $s\cdot x \cdot t = sxt^\phi$ for
$x \in X$, $s \in S$, and $t \in P$.  In the case of $X = { }_SG_S$ and $x \in
G$ with $\phi = c_x$, the map $g \mapsto gx$ gives an isomorphism of $S$-$P$
bisets ${ }_SG_P \cong { }_SG_{c_x}$.  This property determines from fusion
data the information that
\[
u \tr_{S, \psi}^G = (x^{-1}ux)\tr_{S, \psi}^G \quad\quad \text{for $[u,x] \in S$},
\]
i.e. that $S \cap [G,G] \leq \ker(\tr_{S,\psi}^G|_S)$. So $\tr_{S, \psi}^G$
induces a group homomorphism $S/\foc(\F_S(G)) \to A$. 

\begin{definition}\label{D:charbiset}
Let $\F$ be a saturated fusion system on the $p$-group $S$. An $S$-$S$ biset
$\Omega$ is said to be a \emph{characteristic biset} for $\F$ if
\begin{enumerate}
\item[(a)] for each transitive subbiset of $\Omega$ isomorphic to $(S \times
S)/\Delta^{\phi}_P$ as a right $(S \times S)$-set for some $\phi \in
\Hom_\F(P,S)$, 
\item[(b)] for each $P \leq S$ and each $\phi \in \Hom_\F(P,S)$,
the $S$-$P$ bisets ${ }_S\Omega_P$ and ${ }_S\Omega_\phi$ are isomorphic, and
\item[(c)] $|\Omega|/|S|$ is prime to $p$.
\end{enumerate}
\end{definition}

\begin{theorem}[{\cite[Proposition~5.5]{BrotoLeviOliver2003}}]
Each saturated fusion system has a characteristic biset.
\label{T:charbiset}
\end{theorem}

When $\F = \F_S(G)$ for a finite group $G$, the above discussion indicates one
may take $\Omega = G$. A characteristic biset is not uniquely determined by the
fusion system; for instance, a disjoint union of a $p'$ number of copies of
$\Omega$ also satisfies the Linckelmann-Webb properties (a)--(c) if $\Omega$ does. But
for what follows, any characteristic biset will do.

Fix a saturated fusion system $\F$ over the $p$-group $S$, and let $\Omega =
\Omega_\F$ be a characteristic biset for $\F$.  From
\ref{D:charbiset}(a), fix a decomposition of $\Omega$ into
transitive $S$-$S$ bisets:
\[
\Omega = \coprod_{i \in I} (S \times S)/\Delta_{S_i}^{\phi_i}. 
\]
where $S_i \leq S$ and $\phi_i \in \Hom_\F(S_i, S)$ for each $i \in I$.  For a
an abelian group $A$ and a group homomorphism $\psi\colon S \to A$, define the
\emph{transfer map relative to} $\Omega$ to be the homomorphism $\tr_{\Omega,
\psi}: S \to A$ given by
\[
u \tr_{\Omega, \psi} = \prod_{i \in I} u \tr_{S_i, \phi_i\psi}^S.
\]
for $u \in S$, where $\tr_{S_i, \phi_i\psi}^S$ is the ordinary transfer. From
\ref{D:charbiset}(b), 
\begin{align}
\ker(\tr_{\Omega, \psi}) \geq \foc(\F). 
\label{E:LW(b)}
\end{align}

\section{Proof of Theorem~{\rm TL}}
\begin{proof}[of Theorem~{\rm TL}]
Let $\psi\colon S \to S/T$ be the quotient map. Suppose that $u$ has no fully
$\F$-centralized $\F$-conjugate in $T$. We shall show that $u$ lies outside the
kernel of the transfer map $\tr_{\Omega, \psi}$. Once this is done, the Theorem follows
from \eqref{E:LW(b)}.

Without loss of generality we may assume that $u$ itself is fully
$\F$-centralized.  Let $P = C_S(u)$. Applying the Mackey formula for, and the
definition of, ordinary transfer, we have

\begin{align*}
u \tr_{\Omega, \psi} &= \prod_{i \in I} u \tr_{S_i, \phi_i\psi}^S \\ 
&= \prod_{i \in I} \prod_{t \in [P\backslash S/S_i]} u \tr_{\c{t}{S}_i \cap
P,\, c_t\phi_i\psi}^P \\
&= \prod_{i \in I} \prod_{t \in [P \backslash S/S_i]} \prod_{r \in [P/\c{t}{S}_i
\cap P]} (([ur]^{-1}ur)^{c_t\phi_i})\psi
\end{align*}
where $[ur]$ is the representative in $[P/\c{t}{S}_i \cap P]$ corresponding to
the coset $ur(S_i^t \cap P)$.  As $P$ commutes with $u$, we have for an orbit
$\mathcal{O}$ of $\gen{u}$ on the cosets $P/\c{t}{S}_i \cap P$, that $\Pi_{r
\in [\mathcal{O}]} [ur]^{-1}ur = u^{|\mathcal{O}|} \mod T$.  Taking the product
over these orbits, 
\[
u \tr_{\Omega, \psi} = \prod_{i \in I}\prod_{t \in [P \backslash S/S_i]}
(u^{c_t\phi_i})^{|P:\c{t}{S}_i \cap P|} \quad \mod T.
\]
Suppose $i$ and $t$ are such that the index $|P:\c{t}{S}_i \cap P|$ is
divisible by $p$. Then the corresponding factor $(u^{c_t\phi_i})^{|P:S_i^t \cap
P|}$ has order less than that of $u$, so lies in $T$ by assumption and
contributes nothing to the transfer. On the other hand, $|P:\c{t}{S}_i \cap P|
= 1$ if and only if $P \leq \c{t}{S}_i$. In this case, $\varphi_i$ is defined
on $P^t = C_S(u^t)$, and so the corresponding factor $u^{c_t\phi_i}$ is fully
$\F$-centralized. By assumption, $u^{c_t\phi_i} \nin T$. Write $\mathcal IT =
\{u_jT\}_{j \in J}$ with $u_j$ a fully $\F$-centralized $\F$-conjugate of $u$
for each $j \in J$. Let
\[
\mathcal{T} = \{(i, t) \mid i \in I, \,t \in [P\backslash S/S_i], P \leq \c{t}{S}_i\}.
\]
Then by the above remarks,
\[
u \tr_{\Omega, \psi} = \prod_{j \in J} u_j^{k_j}T
\]
with $\sum_{j \in J} k_j = |\mathcal T|$.

We finish by showing that the cardinality of the set $\mathcal T$ is prime
to $p$. Now $P \leq \c{t}{S}_i$ if and only if $P$ fixes the left coset $tS_i$
in its action from the left. Furthermore, $\Omega$ decomposes as a disjoint
union of orbits of the form $S/S_i$ as a right $S$-set.  Therefore $|\mathcal
T| = |(\Omega/S)^P|$, the number of $P$-fixed points in its left action on this
set of orbits. Since $|\Omega/S|$ is prime to $p$ by
\ref{D:charbiset}(c), it follows that $|\mathcal T|$ is also prime
to $p$. As $\sum_{j \in J} k_j = |\mathcal T|$, there exists $j_0 \in J$
with $p \nmid k_{j_0}$. By linear independence of $\{u_jT\}_{j \in J}$, we have
$u \tr_{\Omega, \psi} \neq T$ and so $u \nin \ker(\tr_{\Omega, \psi})$,
completing the proof.  
\end{proof}

When $\F = \F_S(G)$ for some finite group $G$ with Sylow $p$-subgroup $S$, we
can specialize Theorem~TL to obtain a generalization of Lyons' results
\cite[Proposition~15.15]{GLS2} and \cite[Lemma~2.3.1]{GLS6}. 

\begin{corollary}[({cf. \cite[Proposition~15.15]{GLS2}})]\label{C:TTcyclic}
Let $\F$ be a saturated fusion system on a $p$-group $S$ with $\F = O^p(\F)$.
Suppose $T$ is a proper normal subgroup of $S$ with $S/T$ cyclic, and let $u$
be an element of least order in $S - T$.  Assume that every fully
$\F$-centralized $\F$-conjugate of $u$ lies in the coset $T$ or in $uT$. Then
$u$ has a fully $\F$-centralized $\F$-conjugate in $T$.
\end{corollary}

Note that if $p=2$ in Corollary~\ref{C:TTcyclic}, then $uT$ is the unique
involution in the quotient $S/T$, and so the condition that each $\F$-conjugate
of $u$ lies in $T \cup uT$ is automatically satsified. 

\begin{corollary}[({cf. \cite[Lemma~2.3.1]{GLS6}})]\label{C:TTlinind}
Let $\F$ be a saturated fusion system on a finite $2$-group $S$ with
$\F=O^2(\F)$. Suppose $T$ is a proper normal subgroup of $S$ with $S/T$
abelian. Let $\mathcal I$ be the set of fully $\F$-centralized involutions in
$S - T$, and suppose that the set $\mathcal IT = \{vT \mid v \in \mathcal I\}$
is linearly independent in $\Omega_1(S/T)$.  Then each involution $u \in S-T$
has a fully $\F$-centralized $\F$-conjugate in $T$. 
\end{corollary}

\section{Up and down fusion}

We close by pointing out a strengthening of Alperin's fusion theorem for
saturated fusion systems which may be useful when applying Theorem~TL in the
case where one has some knowledge of the structure of $\F$-automorphism groups
of subgroups of $S$ in a given conjugation family $\C$ for $\F$.  This is
motivated by Alperin's observations in \cite{Alperin1974}, and the proof is
his.  

In practice, one applies Theorem~TL in a fusion system with $\F = O^p(\F)$ to
obtain conjugacy information (usually of involutions). The result of applying
the transfer lemma is to get a morphism $\phi\colon C_S(u) \to C_S(u^\phi)$
with $u^\phi \in T$. Sometimes it is advantageous to decompose this morphism
into a composition of restrictions of automorphism groups of subgroups of $S$.
In particular, by Alperin's fusion theorem there will be some conjugate $u_1$ of
$u$, an $\F$-centric subgroup $P$ of $S$, and $\alpha \in \Aut_\F(P)$ such that
$u_1 \in P$, $u_1 \nin T$, and $u_1^\alpha \in T$. This increased control comes
at the cost that $u_1^\alpha$ may no longer be fully $\F$-centralized.
However, Proposition \ref{P:updown} below shows that a decomposition may be
chosen so that the centralizers of intermediate conjugates always ``go up''.
In particular, the above $\alpha$ can always be chosen so that it extends to
$\alpha\colon C_S(u_1) \to C_S(u_1^\alpha)$.

A conjugation family $\C$ for $\F$ is a set of subgroups $Q$ of $S$ such
that every morphism in $\F$ is a composition of restrictions of elements in
$\Aut_\F(Q)$ as $Q$ ranges over $\C$. Examples are the set of fully
$\F$-normalized, $\F$-centric, and $\F$-radical subgroups of $S$, or the set
$\{P \leq S \mid \text{$P$ is $\F$-essential or $P = S$} \}$.  Given an
isomorphism $\varphi \in \Hom_\F(P,P')$, say for short that a sequence $(Q_i,
\alpha_i)_{1\leq i \leq n}$ is an \emph{up} (resp. \emph{down})
$\C$-\emph{decomposition} of $\varphi$ if $Q_i \in \C$ with $\alpha_i \in
\Aut_\F(Q_i)$ and $\phi = \alpha_1\cdots\alpha_n$ (after restriction of the
$\alpha_i$'s) and with $|C_S(P_{i-1})| \leq |C_S(P_i)|$ (resp. $|C_S(P_{i-1})|
\geq |C_S(P_i)|$) for all $1 \leq i \leq n$. (Here, we set $P_0 = P$ and $P_i =
P^{\alpha_1\cdots\alpha_i}$.)

\begin{proposition}\label{P:updown}
Let $\F$ be a saturated fusion system on the $p$-group $S$, and suppose that
$\C$ is a conjugation family for $\F$. Let $P$, $P' \leq S$, let $\phi\colon P
\to P'$ be an isomorphism in $\F$. Then there exists a $\C$-decomposition
$(Q_i, \alpha_i)_{1 \leq i \leq n}$ of $\phi$ together with an integer $0 \leq
k \leq n$ such that if we set $P_i = P^{\alpha_1\cdots\alpha_i}$, then 
\[
|C_S(P)| \leq |C_S(P_1)| \leq \cdots \leq |C_S(P_k)| \geq \cdots \geq |C_S(P_{n-1})|
\geq |C_S(P')|.
\]
\end{proposition}
\begin{proof}
Let $\phi\colon P \to P'$ be an isomorphism in $\F$.  Is suffices to consider the
case when $P'$ is fully $\F$-centralized.  Indeed, if $\psi\colon P' \to P''$ is a
map in $\F$ with $|C_S(P'')|$ maximal and there are up $\C$-decompositions for
$\psi$ and $\phi\psi$, then there is an up-down decomposition for $\phi =
\phi\psi\psi^{-1}$. 

Assume $P'$ is fully $\F$-centralized. Fix an extension $\tilde{\phi}\colon
PC_S(P) \to P'C_S(P')$ of $\phi$ and denote also by $\tilde{\phi}$ the induced
isomorphism $PC_S(P) \to (PC_S(P))^{\tilde{\phi}}$. Let $(Q_i, \alpha_i)_{1
\leq i \leq n}$ be any $\C$-decomposition of $\tilde{\phi}$. Set $P_0 = P$,
$P_i = P_{i-1}^{\alpha_i}$, and $C_0 = C_S(P)$, $C_i = C_{i-1}^{\alpha_i}$, for
$1 \leq i \leq n$. The proof is by induction on the index of $C_S(P)$ in $S$. 

Assume first that $|C_S(P)| = |C_S(P')|$. 
We claim that 
\begin{equation*}
C_i = C_S(P_i) \mbox{ for each $i$ }
\end{equation*}
in this case, so that $(Q_i, \alpha_i)$ is already an up $\C$-decomposition of
$\phi$. The case $i = 0$ holds by definition. If $i \leq n$ and 
$C_{i-1} = C_S(P_{i-1})$, then as $Q_{i} \geq P_{i-1}C_S(P_{i-1})$, we have
\begin{equation*}
C_{i} = C_S(P_{i-1})^{\alpha_{i}} \leq C_S(P_{i}).
\end{equation*}

By induction then
\[
|C_S(P)| = |C_{i-1}| \leq |C_S(P_{i})| \leq |C_S(P')| = |C_S(P)|.
\]
Therefore inequalities are equalities, and $C_{i} = C_S(P_{i})$ as claimed.

Now suppose $|C_S(P)| < |C_S(P')|$ and the proposition holds for all
$\F$-conjugates $R$ of $P'$ with $|S:C_S(R)| < |S:C_S(P)|$. With notation as
before, let $1 \leq l \leq n$ be the smallest integer such that $|C_S(P_{l-1})|
< |C_S(P_l)|$. Then $|C_S(P)| = |C_S(P_1)| = \cdots = |C_S(P_{l-1}|$, since
$C_{S}(P_{i-1})^{\alpha_i} \leq C_S(P_i)$ for each $i$.  By the inductive
hypothesis, there is an up $\C$-decomposition $(Q'_{i}, \alpha'_{i})_{l \leq i
\leq n}$ for $\alpha_{l}\cdots\alpha_{n}: P_l \to P'$, and this completes the
proof of the proposition.  \end{proof}

\begin{remark*}
Proposition~\ref{P:updown} does not hold if centralizers are replaced by the
normalizers in $S$ of the $p$-subgroups. The point is that there need not exist
a decomposition $(Q_i, \alpha_i)_{1 \leq i \leq n}$ of an automorphism $\phi
\in \Aut(P)$ with the property $P_i = P^{\alpha_1\cdots\alpha_i}$ are all fully
$\F$-normalized for $1 \leq i \leq n$. 

As an example, let $H_i \cong S_4$ for $i = 1, 2$ and set $G = H_1 \times H_2$.
Let $S_i \in \Syl_2(H_i)$, $S = S_1 \times S_2$, and $V = O_2(G)$. The
essential subgroups of $\F_S(G)$ are $VS_1$ and $VS_2$. Let $h_i \in
N_{H_i}(VS_{3-i})$ of order $3$, and let $P$ be the four subgroup of $V$ which
is normalized by $h := h_1h_2$ and with $N_S(P)$ of index $2$ in $S$. Then
$|P^{\F}| = 3$ and $N_S(R) = V$ for $R \in P^\F-\{P\}$. Thus, $(VS_{3-i},
c_{h_i})_{1 \leq i \leq 2}$ is the only essential decomposition of $\phi = c_h
\in \Aut_{\F_S(G)}(P)$ and whereas $P$ is fully $\F$-normalized, $P^{h_1}$ is
not.
\end{remark*}

\subsection*{Acknowledgments}
This work forms part of the author's Ph.D. thesis at Ohio State under the
direction of Ron Solomon, and was supported in part by a fellowship from the
Ohio State University Graduate School. The author would also like to thank
Richard Lyons for his comments and suggestions.

\bibliographystyle{plain}{}
\bibliography{/math/home/fac/jl1474/math/research/mybib}

\affiliationone{Justin Lynd\\
                Department of Mathematics\\
                Rutgers University\\
                110 Frelinghuysen Rd\\
                Piscataway, NJ 08854\\
                USA \email{jlynd@math.rutgers.edu}}
                     
\end{document}